\documentclass[12pt]{article}
\usepackage{amsmath}
\usepackage{amsfonts}
\usepackage{amsthm}
\usepackage{amssymb, setspace}
\usepackage{color,graphicx,epsfig,geometry,hyperref,fancyhdr}
\usepackage[T1]{fontenc}
\usepackage{float}
\usepackage{subfig}
\usepackage{bbm}
\usepackage{mathrsfs,fleqn}
\newtheorem{theorem}{Theorem}[section]

\newtheorem{lemma}{Lemma}[section]

\newtheorem{remark}{Remark}[section]

\newcommand{\Z}{\mathbb{Z}}

\newcommand{\R}{\mathbb{R}}
\newcommand{\C}{\mathbb{C}}

\newcommand{\grad}{\nabla}

\newcommand{\ep}{\varepsilon}

\graphicspath{{pics/}}

\begin{document}
%\lhead{}
%\rhead{}

\begin{flushleft}
\Large 
\noindent{\bf \Large Direct sampling for recovering sound soft scatterers from point source measurements }
\end{flushleft}

\vspace{0.2in}

{\bf  \large Isaac Harris}\\
\indent {\small Department of Mathematics, Purdue University, West Lafayette, IN 47907 }\\
\indent {\small Email: \texttt{harri814@purdue.edu}}\\

%\vspace{0.2in}

%%%%%%%%%%%%%%%%%%%%%%%%%%%%%%%%%%%%%%%%
\begin{abstract}
\noindent In this paper, we consider the inverse problem of recovering a sound soft scatterer from the measured scattered field. The scattered field is assumed to be induced by a point source on a curve/surface that is known. Here we will propose and analyze new direct sampling methods for this problem. The first method we consider uses a far-field transformation of the near-field data which will allow us to derives explicit bounds in the resolution analysis for the direct sampling method's imaging functional. Two direct sampling methods will be studied using the far-field transformation. For these imaging functionals we will use the Funk-Hecke identities to study the resolution analysis.We will also study a direct sampling method for the case of the given Cauchy data.  Numerical examples are given to show the applicability of the new imaging functionals for recovering a sound soft scatterer in 2D.
\end{abstract}



%%%%%%%%%%%%%%%%%%%%%%%%%%%%%%%%%%%%%%%%%%%%%%%%%%%%%%%%%%%
\section{Introduction}

Here we develop new direct sampling methods for recovering a sound soft scatterer from the measured scattered field induced by point sources. Direct (also referred to as Orthogonality) sampling methods are qualitative reconstruction methods that have gained interest recently by researchers. These types of reconstruction algorithms where first introduced in \cite{CK}. Just as other qualitative reconstruction methods the direct sampling method requires little a priori information about the specific physical parameters of the scatterer. This implies that these methods are robust in the fact that they recover multiple types of scatterers using the same algorithm(see for e.g. \cite{DSMHarris,postdocpaper,Liu}). Therefore, these methods can be advantageous to use in applications such as non-destructive testing and medical imaging. Similar methods have also been studied in diffuse optical tomography \cite{DSMdot}, electrical impedance tomography \cite{DSMeit} and thermodynamics \cite{DSMheat}. These methods have been studied in detail for far-field data but little has been done for the case of near-field data. In this paper, we will develop and analyze some direct sampling method given near-field measurements.

The main idea behind the methodology of qualitative methods is to develop an imaging functional using the measured data that is positive in the region that you wish to recover and (approximately) zero outside the region. All qualitative methods achieve this in various ways(see for e.g. \cite{TE-book,kirschbook}). One of the main advantages of direct sampling methods is the fact that the imaging functional is usually given by an inner-product(or norm) of the data operator and a specifically chosen function. This implies that these imaging functionals are simple to compute as well as stable with respect to noise in the scattering data. The main analytical tool for studying these methods comes from the factorization of the data operator just as in the factorization method(see for e.g. \cite{FM-wave,FMheat,firstFM}).

The main idea of this paper is to preform a far-field transformation of the near-field operator. This will completely transform the operator to the corresponding far-field operator for the scattering problem. This has the advantage that we can then use the theory already developed in the literature for the far-field operator as well as avoid using the Helmholtz-Kirchhoff identity which is often used in reverse time migration \cite{rtm}. Reverse time migration is very similar to the direct sampling method but in the case of near-field measurements the the Helmholtz-Kirchhoff identity does not provide explicit decay rates for the resolution analysis. For the case of far-field measurements one uses the Funk-Hecke identity. This gives explicit bounds on the imaging functionals as $\mathrm{dist}(z,D) \to \infty$ where $D$ denotes the unknown scatterer to be recovered and $z \in \R^d$ is sampling point where we evaluate the imaging functional. Using the asymptotic bounds on the Bessel functions given in the Funk-Hecke identity we can have theoretical limits on the value of the imaging functional outside the scatterer. We also study a new direct sampling method which uses the measured Cauchy data of the scattered field from point sources.

The rest of the paper is organized as follows. In the next section, we will rigorously describe the scattering problem under consideration as well as derive a factorization for the corresponding near-field operator. The factorization for the near-field operator is critical in the development of the new direct sampling methods. Next, we derive two new direct sampling methods where we use a far-field transformation of the near-field operator. This is done in order to avoid using the Helmholtz-Kirchhoff identity which is often used in reverse time migration. We will also study a direct sampling imaging functional that uses the near-field Cauchy data. For proof of concept we will provide numerical examples for the three imaging functionals studied for both full and partial aperture data. 

%%%%%%%%%%%%%%%%%%%%%%%%%%%%%%%%%%%%%%%%%%%%%%%%%%%%%%%%%%%
\section{Analysis of the scattering problem}\label{sect:dsm}
 In this section, we will derive a factorization of the near-field operator that will be used to analyze the new direct sampling methods imaging functionals for recovering a sound soft scatterer from the measured scattered field. To this end, we begin by formulating the direct time-harmonic scattering problem under consideration. The scattered field denoted by $u^s( \cdot \, , y)$ is induced by a point source incident field $u^i( \cdot \, , y)=\Phi( \cdot \, , y)$. Here we let $y$ denotes the location of the point source located on the curves/surface $\Gamma$ and $\Phi( x , y)$ is the radiating fundamental solution to Helmholtz equation given by 
$$\Phi(x ,y)= \left\{\begin{array}{lr} \frac{\text{i}}{4} H^{(1)}_0 (k | x -y |) \, \, & \, \text{in} \, \, \R^2 \\
 				&  \\
 \frac{1}{4 \pi} \frac{\text{exp}({ \text{i} k | x -y |}) }{| x -y |}  & \,  \text{in} \,\,   \R^3 ,
 \end{array} \right. $$  
for $x \neq y$ where $H^{(1)}_0$ is the first kind Hankel function of order zero. Throughout the paper, we will use boundary integral operators in our analysis so we will assume that $\Gamma$ is a class $\mathcal{C}^2$--smooth closed curves/surface.

Now, let $D \subset \R^d$ (for $d=2,3$) be the sound soft scattering obstacle(possibly with multiple components). We will assume that the boundary $\partial D$ is a class $\mathcal{C}^2$--smooth closed curve/surface where the exterior $\R^d \setminus\overline{D}$ is connected. Therefore, the radiating time-harmonic scattered field $u^s( x , y) \in H^1_{\text{loc}} (\R^d \setminus\overline{D})$ given by the point source incident field is the unique solution to(see for e.g. \cite{CK3})
 \begin{eqnarray}
&&\Delta_x u^s +k^2  u^s=0\,\,  \textrm{ in } \,\, \R^d \setminus\overline{D} \quad \textrm{ and } \quad u^s( \cdot \, , y) = -\Phi( \cdot \, , y) \,\, \textrm{ on } \,\, \partial D \label{scalarprob} \\
&&{\partial_r u^s} - \text{i} ku^s =\mathcal{O} \left( \frac{1}{ r^{(d+1)/2} }\right) \quad \text{ as } \quad r \rightarrow \infty \label{src}
\end{eqnarray} 
where $r=|x|$ with $k$ being the positive wave number. Here we assume that $k^2$ is not a corresponding Dirichlet eigenvalue for the negative Laplacian in $D$.  The Sommerfeld radiation condition given by \eqref{src} is satisfied uniformly in all directions. This gives that we can assume that we have the measured scattering data $u^s(x,y)$ for all $x,y \in \Gamma$ provided that $\text{dist}(\Gamma , D)>0$. Therefore, we now define the so-called near-field operator 
$${N} : L^2(\Gamma) \longmapsto L^2(\Gamma) \quad  \text{given by} \quad  ({N}g)(x) = \int_{\Gamma} u^s(x,y) g(y) \, \text{d}s(y).$$ 
In order to study the {\bf inverse problem} of reconstructing the sound soft scattering obstacle $D$ given the near-field measurements we need to derive a suitable factorization for the near-field operator. In \cite{nf-fm} a factorization of the near-field operator was studied but for our purposes we need to derive a different factorization.

From the direct scattering problem \eqref{scalarprob}--\eqref{src} we make the ansatz that the scattered field can be represented by the boundary integral operator $SL_{\partial D}: H^{-1/2}(\partial D) \longrightarrow H^1_{\text{loc}} (\R^d \setminus\overline{D})$ 
\begin{align}
u^s( \cdot \, , y) = \big(SL_{\partial D}\big) \varphi_y \quad \text{where } \quad \big(SL_{\partial D}\big) \varphi_y= \int_{\partial D} \Phi( \cdot \, ,  \omega ) \varphi_y(\omega) \, \text{d}s(\omega) \label{BIO}
\end{align}
for some $\varphi_y \in H^{-1/2}(\partial D)$. See \cite{mclean} for the mapping properties of the boundary integral operator $SL_{\partial D}$. Therefore, we have that  $\varphi_y$  satisfies the equation
$$\big(S_{\partial D \to \partial D} \big)\varphi_y = -  \Phi( \cdot \, , y) \quad \text{ for any fixed } \,\, y \in C$$
where $S_{\partial D \to \partial D}:H^{-1/2}(\partial D) \longrightarrow H^{1/2}(\partial D)$ is given by 
$$\big(S_{\partial D \to \partial D} \big) \varphi_y =\int_{\partial D} \Phi( \cdot \, ,  \omega) \varphi_y(\omega) \, \text{d}s(\omega) \Big|_{\partial D}.$$
Note that we have used the continuity of the trace for the boundary integral operator $SL_{\partial D}$ on the boundary $\partial D$ (see for e.g. \cite{mclean}). Since $k^2$ is not a Dirichlet eigenvalue of the negative Laplacian in $D$ Lemma 1.14 of \cite{kirschbook} gives that $S_{\partial D \to \partial D}$ has a bounded inverse. We will define the bounded linear operator 
\begin{align}
T:H^{1/2}(\partial D) \longrightarrow H^{-1/2}(\partial D) \quad \text{such that} \quad T= -S^{-1}_{\partial D \to \partial D}. \label{T-def}
\end{align}
This implies that the scattered field has the representation
\begin{align}
u^s( \cdot \, , y) = \int_{\partial D} \Phi( \cdot \, ,  \omega) \big[ T \Phi( \cdot \, , y) \big](\omega) \, \text{d}s(\omega). \label{us-rep}
\end{align}
Equation \eqref{us-rep} gives us an analytical solution to the direct scattering problem using boundary integral operators. One can view \eqref{us-rep} as a stand in for the Lippmann-Schwinger integral representation of the scattered field that one obtains by considering a penetrable scatterer(see for e.g. equation (8.13) of \cite{CK3}). From this, we will derive a factorization of the near-field operator. To this end, define the bounded linear operator 
\begin{align}\label{sl-op1}
S : L^2(\Gamma)  \longmapsto L^2(\partial D) \quad  \text{given by} \quad   S g = \int_{\Gamma}  \Phi( \cdot \, , y)  g(y) \, \text{d}s(y) \Big|_{\partial D} 
\end{align}
along with it's dual-operator 
\begin{align}\label{sl-op2}
S^{\top} : L^2(\partial D)  \longmapsto L^2(\Gamma) \quad  \text{given by} \quad S^{\top} \varphi = \int_{\partial D}  \Phi( \omega , \cdot)  \varphi(\omega) \, \text{d}s(\omega) \Big|_{\Gamma}
\end{align}
with respect to the bilinear $L^2$ dual-product $\langle  \cdot \, , \cdot  \rangle_{L^2}$ such that 
$$\langle  \varphi , Sg  \rangle_{L^2(\partial D)} =\langle  S^{\top} \varphi , g  \rangle_{L^2(\Gamma)} \quad \text{for all } \quad g\in L^2(\Gamma) \textrm{ and } \varphi \in L^2(\partial D).$$
The single layer potential operators $S$ and $S^{\top}$ defined above are commonly used in studying problems in scattering theory. We can now use the operators define above to factorize the data operator $N$.  By superposition, we have that 
$$w(x) = \int_{\Gamma} u^s(x , y) g(y) \text{d}s(y) \quad  \forall\, x \in \R^d \setminus\overline{D}$$
is the scattered field corresponding to \eqref{scalarprob}--\eqref{src}  when the point source incident field is replaced by $Sg$ for some $g \in L^2(\Gamma)$ given by \eqref{sl-op1}. Now, appealing to the representation of the of the scattered field \eqref{us-rep} we have that 
$$w(x) = \int_{\partial D} \Phi( x  ,  \omega) \big[ T Sg \big](\omega) \, \text{d}s(\omega) \quad  \forall\, x \in \R^d \setminus\overline{D}.$$ 
This implies that 
\begin{align}
Ng= w|_{\Gamma}  = S^{\top} \, T \, S g \quad \text{for all } \quad g \in L^2(\Gamma) \label{N-fac}
\end{align}   
by the above definition of the operators in \eqref{sl-op1} and \eqref{sl-op2}. Note that the Range$(S) \subset H^{1/2}(\partial D)$ and $S^\top: H^{-1/2}(\partial D) \longmapsto H^{1/2}(\Gamma)$ by the mapping properties in \cite{mclean}. The factorization of near-field operator \eqref{N-fac} as well as the representation formula \eqref{us-rep} for the scattered field will be instrumental in studying the resolution analysis for the direct sampling imaging functionals presented in the following sections.

 %%%%%%%%%%%%%%%%%%%%%%%%%%%%%%%%%%%%%%%%%%%%%%%%%%%%%%%%%%%
\section{Direct Sampling via Far Field Transform}\label{sect:dsm-fft}
Now that we have the factorization of the near-field data operator given in \eqref{N-fac} we wish to develop new direct sampling methods. To this end, we will use a far-field transformation to the near-field  operator $N$. We will mainly focus on the two dimensional case where as the three dimensional can can be handled similarly. The main advantage is that when considering the direct sampling method for far-field data(see for e.g. \cite{DSMHarris,Liu}) we have that the resolution analysis can easily be obtained by the Funk-Hecke integral identities. Then the resolution can be derived by the asymptotic decay of the Bessel functions. When considering reverse time migration for near-field data(see for e.g. \cite{rtm}) the analysis uses the Helmholtz-Kirchhoff integral identity which does not have an explicit decaying first order terms. This means that analytically one must take the sources/receivers far away from the intended target. Recall, the Funk-Hecke integral identities in \cite{Liu}, are given by 
\begin{equation}\label{F-H-id}
\int_{\mathbb{S}^{d-1}}  \mathrm{e}^{- \mathrm{i} k {(z - x)}  \cdot \hat{y} } \, \mathrm{d}s(\hat{y} ) =\left\{\begin{array}{lr} 2\pi J_0(k | x - z|) \, \, & \quad \text{if} \quad d = 2,\\
 				&  \\
4\pi j_0(k | x - z|) & \quad \text{if} \quad d = 3 
 \end{array} \right. 
\end{equation}
and when $x\neq z$
\begin{equation}\label{F-H-id2}
\int_{\mathbb{S}^{d-1}} \hat{y} \mathrm{e}^{- \mathrm{i} k {(z - x)}  \cdot \hat{y} } \, \mathrm{d}s(\hat{y} ) = \left\{\begin{array}{lr} 2\pi \frac{(x-z)}{\text{i} |x-z|} J_1(k | x - z|) \, \, & \quad \text{if} \quad d = 2,\\
 				&  \\
4\pi \frac{(x-z)}{\text{i} |x-z|} j_1(k | x - z|)  & \quad \text{if} \quad d = 3 
 \end{array} \right. 
\end{equation}
where $\mathbb{S}^{d-1}$ is the unit circle for $d=2$ or unit sphere for $d=3$ i.e. $\mathbb{S}^{d-1}=\{p \in \R^d \, :\, |p|=1\}$. We will make use of the decay of the Bessel functions i.e. 
$$J_0(t) = \frac{\cos t + \sin t}{\sqrt{\pi t}}  \left\{1 + \mathcal{O} \left( \frac{1}{t}\right) \right\} \quad \text{ and } \quad J_1(t) = \frac{\cos t - \sin t}{\sqrt{\pi t}}  \left\{-1 + \mathcal{O} \left( \frac{1}{t}\right) \right\}$$ 
for $d=2$ where as  
$$j_0(t) = \frac{ \sin t}{t}  \left\{1 + \mathcal{O} \left( \frac{1}{t}\right) \right\} \quad \text{ and } \quad j_1(t) = \frac{\cos t}{t}  \left\{-1 + \mathcal{O} \left( \frac{1}{t}\right) \right\}$$ 
for $d=3$ as $t \to \infty$(see for e.g. \cite{Liu}). 

In order to proceed, we need to define the Dirichlet-to-Far-Field Transformation just as in \cite{postdocpaper}. This mapping, takes the Dirichlet data of the radiating exterior Helmholtz equation in the exterior of $\text{Int}(\Gamma)$ to the corresponding far-field pattern. Here, we let $\text{Int}(\Gamma)$ denotes the region inclosed by the collection curve $\Gamma$. Therefore, we have that the Dirichlet-to-Far-Field Transformation $\mathcal{Q}: H^{1/2}(\Gamma) \longrightarrow L^2(\mathbb{S}^{d-1})$ is given by  
\begin{align}
(\mathcal{Q} f)(\hat{x}) = v^{\infty}(\hat{x}), \quad \forall \; \hat{x} \in \mathbb{S}^{d-1} \label{Q-operator}
\end{align}
where $v\in H^1_{\text{loc}}(\R^d \setminus \overline{\text{Int}(\Gamma)})$ is the unique solution to
\begin{align}
\Delta v +k^2 v  =  0 \quad \text{in}  \quad \R^d \setminus \overline{\text{Int}(\Gamma)} \quad \text{ with } \quad v|_\Gamma = f \label{eq-ext1} \\
{\partial_r v} - \text{i} k v =\mathcal{O} \left( \frac{1}{ r^{(d+1)/2} }\right) \quad \text{ as } \quad r \rightarrow \infty.  \label{eq-ext2}
\end{align} 
Now, we define the far-field pattern $v^{\infty}$ where $v$ has the expansion 
$$v(x)= \gamma  \frac{\text{e}^{\text{i}k|x|}}{|x|^{(d-1)/2}} \left\{v^{\infty}(\hat{x}) + \mathcal{O} \left( \frac{1}{|x|}\right) \right\}\; \textrm{  as  } \;  |x| \to \infty$$
where $\hat x:=x/|x|$ (see Chapter 1 of \cite{kirschbook}).  Here the constant 
$$\gamma = \frac{ \mathrm{e}^{\mathrm{i}\pi/4} }{ \sqrt{8 \pi k} } \,\,\, \text{in} \,\,\, \R^2 \quad \text{and} \quad  \gamma = \frac{1}{ 4\pi } \,\,\, \text{in} \,\,\, \R^3. $$
This operator was used in \cite{postdocpaper} to derive a direct sampling method for both isotropic and anisotropic scatterers. In this section, we will see that this can be extended to the case of sound soft scatterers. Also, the operator $\mathcal{Q}$ can be constructed without a priori knowledge of $D$. If $\Gamma$ is a ball centered at the origin with radius $R>0$ i.e. $\Gamma=\partial B(0;R)$ then $\mathcal{Q}$ has an explicit formula via separation of variables given by 
%$$ (\mathcal{Q} f)(\theta) = \int\limits_0^{2\pi} Q (\theta, \phi) f(\phi) \text{d}{\phi} \quad \text{where} \quad Q (\theta, \phi) =  \frac{-2\text{i}}{ \pi \sqrt{2} } \sum_{|m| = 0}^{\infty} \frac{\text{e}^{\text{i}m(\theta - \phi - \pi/2)}}{H^{(1)}_{m}(kR)}$$
$$ (\mathcal{Q} f)(\theta) = \int\limits_0^{2\pi} Q (\theta, \phi) f(\phi) \text{d}{\phi} \quad \text{where} \quad Q (\theta, \phi) = \frac{ (1-\text{i}) }{2 \pi \sqrt{\pi k}} \sum_{|m| = 0}^{\infty} \frac{\text{e}^{\text{i}m(\theta - \phi - \pi/2)}}{H^{(1)}_{m}(kR)}$$
with $f(\phi) :=f\big( R (\cos\phi \, , \, \sin \phi) \big)$, see Section 2 of \cite{postdocpaper}  for details. For our numerical experiments we will truncate the series to approximate the operator $\mathcal{Q}$, which converges geometrically in the operator norm(see \cite{postdocpaper}).

\begin{remark}
When $\Gamma \neq \partial B(0;R)$ we can define Dirichlet-to-Far-Field Transformation $\mathcal{Q}$ by using boundary integral equations. See Section 2 of \cite{kirschbook} for a detailed construction. 
\end{remark}

We will now show that the near-field operator $N$ for a sound soft scatterer can be transformed into the far-field operator(see Chapter 2 \cite{kirschbook}) using the operator $\mathcal{Q}$. Once we have shown this, we can employ similar analysis of the  direct sampling methods studied in \cite{DSMHarris}. Recall, that the single layer potential 
$$ v(x)= \int_{\partial D}\Phi(x , \omega) \varphi(\omega)  \text{d}s(\omega) \quad  \textrm{for any } \quad \varphi \in L^2(\partial D) $$
satisfies \eqref{eq-ext1}--\eqref{eq-ext2} with $f= S^{\top} \varphi$ as defined in \eqref{sl-op2}. Note, that by the mapping properties of the single layer potential(see Chapter 6 of \cite{mclean}) we have that the range of $S^\top$ is a subset of $H^{1/2}(\Gamma)$. Therefore, by the asymptotic expansion of the fundamental solution(see for e.g. \cite{CK3}) we obtain the relationship
\begin{align}
(\mathcal{Q} \, S^{\top} \varphi) (\hat{x}) = \int_{\partial D} \text{e}^{-\text{i}k \hat{x}\cdot \omega} \varphi(\omega)   \text{d}s(\omega)  \quad  \textrm{for any } \quad \varphi \in L^2(\partial D). \label{qs-relation}
\end{align} 
Now, motivated by \eqref{qs-relation} we define the Trace of the Herglotz wave function on the boundary $\partial D$ as the bounded linear operator $H: L^2( \mathbb{S}^{d-1}) \longrightarrow L^2( \partial D)$ given by 
$$(H g)(\omega) = \int_{\mathbb{S}^{d-1}} \text{e}^{\text{i}k  \omega \cdot \hat{x}} g(\hat{x}) \mathrm{d} s(\hat{x}) \Big|_{\partial D}  \quad  \textrm{for any } \quad g \in L^2(\mathbb{S}^{d-1}).$$
Now, we recall the adjoint operator $H^*: L^2( \partial D ) \longrightarrow L^2(\mathbb{S}^{d-1})$ such that 
$$(Hg , \varphi)_{ L^2(\partial D)} = (g , H^* \varphi)_{L^2( \mathbb{S}^{d-1})} \quad \text{for all } \quad g\in L^2( \mathbb{S}^{d-1}) \textrm{ and } \varphi \in L^2(\partial D)$$
which is given by 
$$ (H^* \varphi) (\hat{x}) =  \int_{\partial D} \text{e}^{-\text{i}k \hat{x}\cdot \omega} \varphi(\omega)   \text{d}s(\omega)  \quad  \textrm{for any } \quad \varphi \in L^2(\partial D)$$ 
(see for e.g. \cite{kirschbook} Chapter 2). Therefore, by appealing to \eqref{qs-relation} we have that 
$$\mathcal{Q} \, S^{\top} \varphi = H^* \varphi \quad \textrm{for any } \quad \varphi \in L^2(D).$$
From this, we have that dual-operator of $(H^*)^\top : L^2( \mathbb{S}^{d-1}) \longrightarrow L^2( \partial D)$ satisfies 
$$S\mathcal{Q}^\top = (H^*)^\top \quad  \textrm{ where the operator} \quad \left( (H^*)^\top g \right) (\omega) := \int_{\mathbb{S}^{d-1}} \text{e}^{- \text{i}k  \omega \cdot \hat{x}} g(\hat{x}) \mathrm{d} s(\hat{x}) \Big|_{\partial D}$$ 
for all $g \in L^2(\mathbb{S}^{d-1})$(see for e.g. Chapter 2 of \cite{Brezis}). To continue, we notice that 
\begin{align*}
\left( (H^*)^\top g \right) (\omega) &= \int_{\mathbb{S}^{d-1}} \text{e}^{- \text{i}k  \omega \cdot \hat{x}} g(\hat{x}) \mathrm{d} s(\hat{x}) \Big|_{\partial D}\\
						   &= \int_{\mathbb{S}^{d-1}} \text{e}^{\text{i}k  \omega \cdot \hat{x}} g(-\hat{x}) \mathrm{d} s(\hat{x}) \Big|_{\partial D}\\
						   &=(H\mathcal{R} g)(\omega)
\end{align*}
where the operator $\mathcal{R}:  L^2( \mathbb{S}^{d-1}) \longrightarrow L^2(\mathbb{S}^{d-1})$ is given by $(\mathcal{R}g)(\hat{x}) = g(-\hat{x})$. It is clear that $\mathcal{R}$ is a bounded linear operator and $\mathcal{R} = \mathcal{R}^{-1}$. By the definition of the operators $\mathcal{Q}$ and $\mathcal{R}$ we can conclude that 
\begin{align} \label{FFT}
\mathcal{Q} N \mathcal{Q}^\top \mathcal{R} =H^*TH \quad \text{ where } \quad \mathcal{Q} N \mathcal{Q}^\top \mathcal{R} :  L^2( \mathbb{S}^{d-1}) \longrightarrow L^2(\mathbb{S}^{d-1})
\end{align}
by appealing to the factorization in \eqref{N-fac}. Note, by equation (1.55) in \cite{kirschbook} we can conclude that $\mathcal{Q} N \mathcal{Q}^\top \mathcal{R}$ corresponds to the far-field operator for the scattering problem \eqref{scalarprob}--\eqref{src} where the incident field is a plane wave. 
%This implies that similar analysis as in \cite{DSMHarris} can be applied to $\mathcal{Q} N \mathcal{Q}^\top \mathcal{R}$.

{\bf The imaging functional via Far-Field transform:} We now have all we need to define two new imaging functionals via for the transformed operator $\mathcal{Q} N \mathcal{Q}^\top \mathcal{R}$. For each sampling point $z \in \R^d$ the imaging functional via the far-field transform is given by 
\begin{align}
W_{\text{FF}}(z) = \Big|\Big(\mathcal{Q} N \mathcal{Q}^\top \mathcal{R} \,  \phi_z, \phi_z \Big)_{L^2(\mathbb{S}^{d-1})}\Big| \quad  \textrm{ where } \quad  \phi_z(\hat{x}) = \text{e}^{-\text{i}k z \cdot \hat{x}}. \label{dsm1}
\end{align} 
Due to the fact that $\mathcal{Q} N \mathcal{Q}^\top \mathcal{R}$ transforms the near-field operator $N$ for the scattering problem \eqref{scalarprob}--\eqref{src} to the corresponding far-field operator we can appeal to the results in \cite{DSMHarris}. 

\begin{theorem}\label{decay1}
Let the imaging functional $W_{\text{FF}}(z)$ be as defined by \eqref{dsm1}. Then for any sampling point $z \in \R^d \setminus \overline{D}$ we have that  
$$W_{\text{FF}}(z)  =  \mathcal{O} \left(\mathrm{dist}(z,D)^{1-d}\right)  \quad \text{ as } \quad\mathrm{dist}(z,D) \to \infty.$$
\end{theorem}
\begin{proof} 
In order to prove the result we let $v_g$ denote the Herglotz wave function for any $x \in \R^d$ given by 
$$v_g(x) =  \int_{ \mathbb{S}^{d-1} }  \mathrm{e}^{\mathrm{i} k {x}  \cdot \hat{y} } g(\hat{y} ) \, \mathrm{d}s(\hat{y} ) \,.$$
Here we can see that $v_g \in H^1_{\text{loc}}(\R^d)$ for any given $g \in L^2(\mathbb{S}^{d-1})$ which then implies that $v_g |_{\partial D} = Hg \in H^{1/2}(\partial D)$.
Therefore, just as in \cite{DSMHarris} we see that for any $z \in \R^d$
\begin{align*}
\Big|\Big(\mathcal{Q} N \mathcal{Q}^\top \mathcal{R} \,  \phi_z, \phi_z \Big)_{L^2(\mathbb{S}^{d-1})}\Big|  &= \left| ( TH \phi_z , H \phi_z )_{L^2(\partial D)}  \right|  \quad \text{by equation \eqref{FFT}} \\
									  &\leq C\| H \phi_z \|_{H^{1/2}(\partial D)}^2  \quad \text{by the boundedness of $T$}\\
 									  &= C \| v_{\phi_z} \|^2_{H^{1/2}(\partial D)} \quad \text{by the definition of $v_g$}\\
									  &\leq C \| v_{\phi_z} \|^2_{H^1(D)} \quad \text{by the Trace Theorem(see for e.g. \cite{evans})}.
\end{align*}
Now, by the definition of Herglotz wave function and the Funk-Hecke integral identities \eqref{F-H-id}--\eqref{F-H-id2} we have that 
$$  \| v_{\phi_z} \|^2_{H^1(D)} = \mathcal{O} \left(\mathrm{dist}(z,D)^{1-d}\right) \quad \text{ as } \quad\mathrm{dist}(z,D) \to \infty$$ 
where we have used the decay of the Bessel functions. 
\end{proof}
The result in Theorem \ref{decay1} is the same as in Theorem 1 in \cite{DSMHarris}. This is due to the fact that we have transformed the near-field operator into the far-field operator. A similar construction was considered in Section 2.4 of \cite{kirschbook} where boundary integral operators are used. Here we avoid the complex computational set up by not appealing to boundary integral operators to define the imaging functional when $\Gamma$ is a circle/sphere. From Theorem 2.8 of \cite{Liu} we have that the imaging functional $W_{\text{FF}}(z)$ is stable with respects to perturbations in the operator $N$.

By further appealing to the results in \cite{DSMHarris} we can construct another direct sampling imaging functionals using the transformed operator $\mathcal{Q} N \mathcal{Q}^\top \mathcal{R}$. We note that the analysis in Chapter 1 of \cite{kirschbook} gives that the compact operator $\mathcal{Q} N \mathcal{Q}^\top \mathcal{R}$ is injective and has a complete orthonormal eigensystem in $L^2(\mathbb{S}^{d-1})$ provided that $k^2$ is not a Dirichlet eigenvalue of the negative Laplacian in $D$. Let $(\lambda_j , \psi_j) \in \C \setminus \{0\} \times L^2(\mathbb{S}^{d-1})$ be the orthonormal eigensystem for the operator  $\mathcal{Q} N \mathcal{Q}^\top \mathcal{R}$. Therefore, we can define 
$$ |\mathcal{Q} N \mathcal{Q}^\top \mathcal{R} |^{p} g =  \sum\limits_{j=1}^{\infty}{ |\lambda_j |}^p (g,\psi_j)_{_{L^2(\mathbb{S}^{d-1})}} \psi_j   $$
for any fixed $p>0$ where the set $\{ \psi_j\}$ is an orthonormal basis in $L^2(\mathbb{S}^{d-1})$. Then, we have that the Factorization Method can be used to recover the scatterer $D$ (see for e.g. \cite{firstFM}) which gives the result that  
\begin{align}
\big|\mathcal{Q} N \mathcal{Q}^\top \mathcal{R} \big|^{1/2} g_z =\phi_z \quad \text{ for } z \in \R^d \label{fm-equ}
\end{align}
is solvable if and only if the sampling point $z \in D$. In \cite{DSMHarris} a connection between the Tikhonov regularized solution to \eqref{fm-equ} was used to developed another imaging functional. To derive the new imaging functional define the Tikhonov filter function for \eqref{fm-equ} given by 
$$\Gamma_{\alpha} (t)= \frac{ \sqrt{t}}{\alpha + t} \quad \text{ on the interval} \quad \left[0, \big\|\mathcal{Q} N \mathcal{Q}^\top \mathcal{R} \big\| \right]$$
which is a continuous function on the given interval. Here, we let $\alpha>0$ denote the {\it fixed} regularization parameter. See \cite{kirschipbook} for the study of regularization techniques and the Factorization Method. From this, we have that for every $\ep >0$ there is an approximating polynomial $ P_{\alpha,\ep}(t)$ such that 
\begin{align}
\|P_{\alpha,\ep}(t) -  \Gamma_{\alpha} (t) \|_{L^{\infty}} < \ep \quad \text{ on the interval} \quad \left[0, \big\|\mathcal{Q} N \mathcal{Q}^\top \mathcal{R} \big\| \right]. \label{poly}
\end{align}
Using the approximating polynomial $P_{\alpha,\ep}(t)$ for the filter function we can define a new imaging functional. 

{\bf The imaging functional via Tikhonov Regularization:} This imaging functional derived from the Tikhonov regularization of \eqref{fm-equ} for {\it fixed} regularization parameter $\alpha>0$ and approximation error $\ep$ is given by 
\begin{align}
W_{\text{TDSM}}(z) = \left\| {\color{black} P_{\alpha,\ep}}\left( \big|\mathcal{Q} N \mathcal{Q}^\top \mathcal{R} \big| \right) \phi_z \right\|^2_{L^2(\mathbb{S}^{d-1})} \quad  \textrm{ where } \quad  \phi_z(\hat{x}) = \text{e}^{-\text{i}k z \cdot \hat{x}}. \label{dsm2}
\end{align}
where the polynomial $P$ satisfies $P_{\alpha,\ep}(t) =\Gamma_{\alpha} (t) + \mathcal{O}(\ep)$ as $\ep \to 0$. From this we have the following result by the analysis in \cite{DSMHarris}. The imaging functional given in \eqref{dsm2} was motivated by the work in \cite{dsm-fm} and we extend that work to the case of near-field data. There are some interesting questions when considering the implementation of the $W_{\text{TDSM}}(z)$ such as: how to pick $\alpha$ and which polynomial approximation method works best for constructing $P_{\alpha,\ep}(t)$. Even with these unanswered questions our numerical experiments show that $W_{\text{TDSM}}(z)$ can provide good restrictions of the scatterer for a simple least-squares polynomial approximation and without having to find an optimal regularization parameter.

\begin{theorem} \label{TDMS}
Let the imaging functional $W_{\text{TDSM}}(z)$ be as defined by \eqref{dsm2}. Then for any sampling point $z \in \R^d \setminus \overline{D}$ we have that  
$\exists \, C_\alpha>0$ independent of $z$ such that 
$$W_{\text{TDSM}}(z) \leq C_\alpha W_{\text{FF}}(z) + \mathcal{O}( \ep ) \quad  \text{ as } \quad \ep \to 0$$
%$$W_{\text{TDSM}}(z) =  \mathcal{O} \left(\mathrm{dist}(z,D)^{1-d}\right)+ \mathcal{O}( \ep ) \quad  \text{ as } \quad \ep \to 0 \quad  \text{ and } \quad\mathrm{dist}(z,D) \to \infty$$
for all fixed $\alpha>0$ provided that $P_{\alpha,\ep}(t) =\Gamma_{\alpha} (t) + \mathcal{O}(\ep)$ as $\ep \to 0$.
\end{theorem}
\begin{proof}
For the proof of Theorem \ref{TDMS} see Section 4 of \cite{DSMHarris} to avoid repetition. 
\end{proof}

The imaging functional provided in \eqref{dsm2} is equivalent to the one studied in \cite{DSMHarris} where one has the far-field operator which corresponds to $\mathcal{Q} N \mathcal{Q}^\top \mathcal{R}$. We also note that even thought the stability of the imaging functional $W_{\text{TDSM}}(z)$ has not been established our numerical experiments with added noise in the data still provides good reconstructions. Also, the numerical experiments provided in \cite{DSMHarris} it has been seen that $W_{\text{TDSM}}(z)$ provides better reconstructions than $W_{\text{FF}}(z)$ in $\R^3$. This has not been verified theoretically but the multiple examples in \cite{DSMHarris} would seem to suggest this to be true. 

Notice, that just like the Dirichlet-to-Far-Field Transformation $\mathcal{Q}$ we can construct the operator $\mathcal{R}$ without a priori knowledge of $D$. We will also see that, $\mathcal{R}$ can easily and efficiently be approximated numerically. To do so, we will right the operator as an integral operator with an explicit kernel function. From the fact that $\hat x =(\cos\theta \, , \, \sin \theta)$ for $\theta \in [0, 2\pi)$ if $d=2$ we let $g(\theta) =g\big( (\cos\theta \, , \, \sin \theta) \big)$. This implies that 
\begin{align}
(\mathcal{R}g)(\theta) = g(\theta+\pi) \label{r-def}
\end{align}
where we have used the sum of angles formula to obtain the equalities 
$$-\cos(\theta)=\cos(\theta+\pi) \quad \text{and } \quad -\sin(\theta)=\sin(\theta+\pi).$$
 Now, using the Fourier series for  $g(\theta)$ we have that  
$$g(\theta) = \sum_{|m| = 0}^{\infty} {g}_{m} \text{e}^{\text{i}m\theta} \quad \text{ where } \quad  g_m =  \frac{1}{2\pi} \int\limits_0^{2\pi} g(\phi) \text{e}^{-\text{i} m \phi} \text{d} \phi \quad \text{ for all} \,\, m\in \mathbb{Z}.$$ 
Therefore, by the definition of the operator $\mathcal{R}$ in equation \eqref{r-def} we obtain that 
$$ (\mathcal{R} g)(\theta) = \int\limits_0^{2\pi} R (\theta, \phi) g(\phi) \text{d}{\phi} \quad \text{where} \quad R(\theta, \phi) = \frac{1}{2\pi} \sum_{|m| = 0}^{\infty} \text{e}^{\text{i}m(\theta - \phi + \pi)}$$
by using \eqref{r-def} as well as the Fourier series for $g$. 
In order to numerically compute the direct sampling methods imaging functionals we need a way to compute $(\mathcal{R} g)(\theta)$. To this end, we know give a result that implies that $\mathcal{R}$ can be approximated by a truncated series for sufficiently smooth $g$. 

\begin{lemma}\label{R-converge}
Let $\mathcal{R}: H^{p}(0,2\pi) \longrightarrow L^2(0,2 \pi)$ be the operator defined by \eqref{r-def} and $\mathcal{R}_M: H^{p}(0,2\pi) \longrightarrow L^2(0,2 \pi)$ be the truncated series for some $M \in \mathbb{N}$ with $p>0$. Then we have norm-convergence with convergence rate given by
$$\|\mathcal{R} - \mathcal{R}_{M}\|_{H^{p}(0,2\pi) \mapsto L^2(0,2 \pi)} =\mathcal{O}\left( \frac{1}{ M^p} \right) , \quad \text{as} \quad M \longrightarrow \infty.$$
\end{lemma}
\begin{proof}
To begin, we clearly see that 
$$(\mathcal{R} - \mathcal{R}_{M})g =  \frac{1}{2\pi} \sum_{|m| = M+1}^{\infty}g_m \text{e}^{\text{i}m (\theta+\pi)}.$$
To prove the claim, we now estimate the $L^2(0,2\pi)$--norm of $(\mathcal{R} - \mathcal{R}_{M})g$ which is given by 
$$\|(\mathcal{R} - \mathcal{R}_{M})g\|^2_{L^2(0,2\pi) } = \sum_{|m| = M+1}^{\infty}  | \text{e}^{\text{i}m \pi} g_m |^2 .$$
By using the fact that $| \text{e}^{\text{i}m \pi}| =1$ for any $m \in \Z$ we have that  
\begin{align*}
\|(\mathcal{R} - \mathcal{R}_{M})g\|^2_{L^2(0,2\pi) } &= \sum_{|m| = M+1}^{\infty}  \frac{(1+|m|^2)^p}{(1+|m|^2)^p} |g_m |^2 \\
									      &\leq \frac{1}{(1+M^2)^p}  \sum_{|m| = M+1}^{\infty}  (1+|m|^2)^p |g_m |^2 \\
									      &\leq  \frac{1}{M^{2p}} \|g\|^2_{H^{p}(0,2\pi)}.
\end{align*}
Taking the supremum over $g$ with unit norm in $H^{p}(0,2\pi)$ proves the claim. 
\end{proof}

\begin{remark}
Even though we only focus on the analysis in $\R^2$, it is clear that one can define $\mathcal{R}$ similarly in $\R^3$. In $\R^3$ one can use Spherical Harmonics to define the operator as well as prove a similar approximation result as in Lemma \ref{R-converge}.
\end{remark}

Recall, that in order to compute the imaging functionals $W_{\text{FF}}(z)$ and $W_{\text{TDSM}}(z)$ given by equation \eqref{dsm1} and  \eqref{dsm1} respectively, we need to evaluate $\mathcal{R}\phi_z$. Since $\phi_z$ is given by a  plane wave we have that it is a  smooth function.
This implies that $\mathcal{R}_M \phi_z \approx \mathcal{R}\phi_z$ where the truncation $M$ can be taken to be a reasonably small. In Section \ref{numerics}, we see that  $M=10$ gives good reconstructions of the scatterer $D$ when used to truncate the operators $\mathcal{R}$ and $\mathcal{Q}$ for both of the imaging functionals studied in this section.

%%%%%%%%%%%%%%%%%%%%%%%%%%%%%%%%%%%%%%%%%%%%%%%%%%%%%%%%%%%
\section{Direct Sampling with Cauchy Data}\label{sect:dsm-cauchy}
In this section, we consider another direct sampling method imaging functional where one has access to the Cauchy data for the scattering problem \eqref{scalarprob}--\eqref{src}. The main idea is to use the representation of the solution given in \eqref{us-rep} to establish the resolution analysis. The imaging functional we consider has been studied for the case when the scatterer is either isotropic or anisotropic in \cite{postdocpaper}. The analysis of the imaging functional studied here has not been done for the case of a sound soft scatterer, which is the case in this paper. Therefore, in this section, we will assume that we have the `measured' Cauchy data.  

\begin{remark}
Notice, that if  only the scattered field $u^s(x,y)$ is given for all $x,y \in \Gamma$ then one can compute the normal derivative $\partial_{\nu} u^s(x,y)$ on $\Gamma$. This can be done by computing the scattered field on the exterior of Int$(\Gamma)$ by using a similar formulation as in \eqref{BIO}. 
\end{remark}

{\bf The imaging functional for Cauchy data:} Assume that we have the Cauchy data $u^s(x,y)$ and $\partial_{\nu} u^s(x,y)$ for all  $x,y \in \Gamma$ that corresponds to the scattering problem \eqref{scalarprob}--\eqref{src}. Then, we define a the following imaging functional
\begin{align}
W_{\text{CD}}(z) = \int_{\Gamma} \left | \int_\Gamma \partial_{\nu} \overline{ \Phi(x,z)} u^s(x,y) - \overline{\Phi(x,z)} \partial_{\nu}u^s(x,y)  \mathrm{d}s(x)\right |^\rho  \mathrm{d}s(y) \label{dsm3}
\end{align}
where $\rho>0$ is a positive constant. Here, again $\Phi$ corresponds to the radiating fundamental solution to Helmholtz equation. Also, the normal derivative is with respect the to $x$ variable such that $ \partial_{\nu} = \nu(x) \cdot \grad_x$ for any $x \in \Gamma$. 

In order to analyize the imaging functional $W_{\text{CD}}(z)$ we first recall that by \eqref{us-rep} we have that the scattered field is given by  
$$u^s( x, y) = \int_{\partial D} \Phi( x ,  \omega) \big[ T \Phi( \cdot \, , y) \big](\omega) \, \text{d}s(\omega) \quad \text{ for any} \quad x,y \in \Gamma$$
where $T:H^{1/2}(\partial D) \longrightarrow H^{-1/2}(\partial D)$ is a bounded linear operator, provided that $k^2$ is not a Dirichlet eigenvalue of the negative Laplacian in $D$. We will derive an equivalent expression for the imaging functional $W_{\text{CD}}(z)$ where the dependance on $\partial D$ is made more explicit. To this end, by taking the normal derivative on $\Gamma$ of the above representation of the scattered field we have that 
$$\partial_{\nu} u^s( x, y) = \int_{\partial D}\partial_{\nu} \Phi( x ,  \omega) \big[ T \Phi( \cdot \, , y) \big](\omega) \, \text{d}s(\omega) \quad \text{ for any} \quad x,y \in \Gamma. $$
Notice, that by using the above representations of the Cauchy data for $u^s(x,y)$ and $\partial_{\nu} u^s(x,y)$ we have that 
\begin{align*}
& \int_\Gamma \partial_{\nu} \overline{ \Phi(x,z)} u^s(x,y) - \overline{\Phi(x,z)}\partial_{\nu} u^s(x,y)  \text{d}s(x)  \\
 & \hspace{0.2in}=  \int_\Gamma  \left[ \partial_{\nu}  \overline{\Phi(x,z)}  \int_{\partial D} \Phi(x , \omega) T \Phi(  \cdot \, ,y)  \mathrm{d}s(\omega) - \overline{ \Phi(x,z)}   \int_{\partial D} \partial_{\nu}  \Phi(x , \omega) T \Phi(  \cdot \, ,y)  \mathrm{d}s(\omega) \right]  \text{d}s(x) \\
 & \hspace{0.2in} =  \int_\Gamma  \int_{\partial D}  \left[ \partial_{\nu}  \overline{ \Phi(x,z)} \Phi(x,\omega) -  \overline{ \Phi(x,z)} \partial_{\nu} \Phi(x,\omega)\right]  T \Phi(  \cdot \,,y) \text{d}s(\omega) \,  \text{d}s(x)\\
 & \hspace{0.2in}= \int_{\partial D} \left[ \int_\Gamma \partial_{\nu}  \overline{ \Phi(x,z)} \Phi(x,\omega) -  \overline{ \Phi(x,z)} \partial_{\nu} \Phi(x,\omega)  \text{d}s(x) \right] T \Phi(  \cdot \, ,y)  \text{d}s(\omega).
\end{align*}
From the analysis in Section 2.2 of \cite{postdocpaper} we have that 
$$ \frac{1}{2\text{i}} \Im \Phi(z,\omega) =  \int_\Gamma   \left[ \partial_{\nu}  \overline{ \Phi(x,z)} \Phi(x,\omega) -  \overline{ \Phi(x,z)} \partial_{\nu} \Phi(x,\omega) \right] \text{d}s(x).$$
This is a simple consequence of Green's second identity and the symmetry of the fundamental solution. Therefore, we have that the imaging functional is equivalent to 
$$W_{CD}(z) = \int_{\Gamma} \left |\int_{\partial D} \frac{1}{2\text{i}} \Im \Phi(z,\omega)  [T \Phi(  \cdot \, ,y)](\omega)  \text{d}s(\omega) \right |^\rho \text{d}s(y)$$
where  we recall that 
$$\Im \Phi(z, \omega) = \left\{\begin{array}{lr} \frac{1}{4} J_0(k | \omega - z|) \, \, & \quad \text{if} \quad d = 2  \,,\\
 				&  \\
\frac{1}{4\pi} j_0(k | \omega - z|) & \quad \text{if} \quad d = 3 \,.
\end{array} \right. $$
As we see by the equivalent representation of $W_{CD}(z)$ we have that the inner integral is in-terms of a Bessel function kernel which will be maximal when $z$ is on $\partial D$ but will decay as the sampling point moves away from the scatterer. 

\begin{theorem}\label{decay2}
Let the imaging functional $W_{\text{CD}}(z)$ be as defined by \eqref{dsm3}. Then for any sampling point $z \in \R^d \setminus \overline{D}$ we have that  
$$W_{\text{CD}}(z)  =  \mathcal{O} \left( \text{dist}(z,D)^{(1-d)\rho/2} \right) \quad \text{ as } \,\,\, \text{dist}(z,D) \to \infty \; \textrm{ for } \; d = 2, 3.$$
\end{theorem}
\begin{proof}
To begin, we first recall $T:H^{1/2}(\partial D) \longrightarrow H^{-1/2}(\partial D)$ is a bounded linear operator, provided that $k^2$ is not a Dirichlet eigenvalue of the negative Laplacian in $D$. Therefore, we have the estimates  
\begin{align*}
W_{CD}(z)   &= \int_{\Gamma} \left |\int_{\partial D} \frac{1}{2\text{i}} \Im \Phi(z,\omega)  [T \Phi(  \cdot \, ,y)](\omega)  \text{d}s(\omega) \right |^\rho  \text{d}s(y)  \\
		   &\leq C  \|\Im \Phi(z,\cdot) \|^\rho _{H^{1/2}(\partial D)} \int_{\Gamma} \|T \Phi(\cdot \, , y) \|^\rho _{H^{-1/2}(\partial D)} \text{d}s(y)  \quad \text{by the dual--pairing }\\
 		   &\leq C  \|\Im \Phi(z,\cdot) \|^\rho _{H^{1/2}(\partial D)} \int_{\Gamma} \| \Phi(\cdot \, , y) \|^\rho _{H^{1/2}(\partial D)} \text{d}s(y)  \quad \text{by the boundedness of $T$}\\
		   &\leq C  \|\Im \Phi(z,\cdot) \|_{H^{1}(D)} \int_{\Gamma} \| \Phi(\cdot \, , y) \|^\rho _{H^{1}(D)} \text{d}s(y) \quad \text{by the Trace Theorem}.
\end{align*}
Notice that $\Phi(\cdot \, , y)$ restricted to the scatterer $D$ is a smooth function for every $y \in \Gamma$ since we have assumed that $\text{dist}(\Gamma , D)>0$. Then, we have obtained that 
$$\int_\Gamma \| \Phi(\cdot \, , y) \|^\rho _{H^{1}(D)} \text{d}s(y) =\| \Phi \|^\rho _{L^{\rho}\left[\Gamma ; H^1(D) \right]} $$
which is a fixed constant depending on $D$ and $\Gamma$. Therefore, we have that 
$$W_{CD}(z) \leq C  \|\Im \Phi(z,\cdot) \|^\rho _{H^{1}(D)}$$
and we then use the fact that 
$$  \|J_0(k|\cdot - z|) \|_{H^{1}(D)} =\mathcal{O}\big( \text{dist}(z,D)^{-1/2}\big) \quad \text{and} \quad \|j_0(k|\cdot - z|) \|_{H^{1}(D)} =\mathcal{O}\big( \text{dist}(z,D)^{-1}\big)$$
as $\text{dist}(z,D) \to \infty$, which proves the claim.
\end{proof}

Note, that the imaging functional in this section does not require a transformation of the data as in the previous section but one does need both pieces of Cauchy data. In order to have an explicit decay rate for the direct sampling functionals as $\mathrm{dist}(z,D) \to \infty$ one must either transform the near-field data or use the Cauchy data in the reconstruction. This is due to the fact that when deriving bounds for the imaging functionals using just the scattered field as in reverse time migration one uses the Helmholtz-Kirchhoff integral identity. This integral identity has a leading order term that is bounded as $\mathrm{dist}(z,D) \to \infty$. This is in contrast to the Funk-Hecke integral identity which can be evaluated explicitly using Bessel functions which has a well established asymptotic decay rate.

\section{Numerical Examples}\label{numerics}
In this section, we present some numerical examples in $\R^2$ to show the applicability of the imaging functionals studied in the previous sections. In our examples, we will take the scattering obstacle $D$ to be a star-like region with respect to the origin for simplicity. Therefore, we will take the the boundary of the scatterer to be given by 
$$\partial D = r(\theta)\big( \cos \theta \, , \sin \theta) \quad \text{for all}  \quad 0\leq \theta \leq 2\pi$$
where the radial function $r(\theta)>0$ is a smooth $2\pi$ period function. 
The radial functions we consider in our examples are given by 
\begin{eqnarray*}
r(\theta) &=&0.5 \quad \text{circular domain,} \\
r(\theta) &=&0.25 \big(2+0.5\cos(3\theta) \big) \quad \text{acorn-shaped domain.} \\
r(\theta) &=&0.75\big(1-0.25\sin(4\theta)\big) \quad \text{flower-shaped domain,} \\
r(\theta) &=&0.5 \Big( |\sin(\theta)|^{10} +  \frac{1}{10} |\cos(\theta)|^{10} \Big)^{-1/10} \quad \text{rounded-square domain.} 
\end{eqnarray*}

Here, we will take $\Gamma=\partial B(0;R)$ with $R=5$ in all our examples. The locations of the sources are given by $y_j = 5 \big( \cos \theta_j \, , \sin \theta_j)$ for 64 equally spaced points  $\theta_j \in [0 , 2\pi )$. In order to compute the simulated scattering data we will assume that the scattered field has the series representation in $\R^2 \setminus \overline{D}$ such that 
$$u^s(x,y_j) = \sum\limits_{|m|=0}^{\infty} \alpha_m(y_j) {H}^{(1)}_m (k|x|) \text{e}^{\text{i}m\theta_x} \quad \text{ for each } \quad y_j \in \Gamma$$ 
where $ {H}^{(1)}_m$ is the first kind Hankel function of order $m$. The coefficients (depending only on $y_j$) $\alpha_m$ in the series representation are determined by the boundary condition on $\partial D$. To compute the scattering data we solve $u^s( \cdot \, , y_j) = -\Phi( \cdot \, , y_j )$ on $\partial D$ for the truncated series where $|m|=0, \, \cdots, 15$ on the discretized boundary $r(\theta_{x_i})\big( \cos \theta_{x_i} \, , \sin \theta_{x_i})$ where $\theta_{x_i} \in [0 , 2\pi )$ are 64 equally spaced point. 
This implies that each $j$ gives a $64 \times 31$ linear system for the coefficients $\alpha_m$ which is solved via spectral cut-off since the resulting matrix is highly ill-conditioned. Once the coefficients have been computed we have that the approximation of the scattering data on $\Gamma$ is given by 
$$u^s(x_i,y_j) \approx \sum\limits_{|m|=0}^{15} \alpha_m(y_j) \text{H}^{(1)}_m (5k) \text{e}^{\text{i}m\theta_{x_i}} $$
and 
$$\partial_{\nu}  u^s(x_i,y_j) \approx \sum\limits_{|m|=0}^{15}\alpha_m(y_j)  \frac{k}{2}   \big( {H}^{(1)}_{m-1} (5k) - {H}^{(1)}_{m+1} (5k) \big) \text{e}^{\text{i}m\theta_{x_i}}$$ 
where we have used the recursive relations for Bessel functions to compute the derivative of Hankel functions. In many applications, the measured scattering data is given with random noise, so we let $\delta$ denote the noise level. Therefore, we have that in our simulations the noisy data is given by 
$$u^{s,\delta} (x_i , y_j)= u^s(x_i , y_j) \left(1 + \delta E_{i,j}\right) \quad \text{and} \quad  \partial_{\nu} u^{s,\delta} (x_i , y_j) =  \partial_{\nu} u^{s} (x_i , y_j) \left(1 + \delta E_{i,j}\right)$$
where $E$ is the random complex--valued matrix of size $64 \times 64$ with norm $\|E\|_{2} = 1$. 

In order to compute the imaging functional $W_{\text{FF}}(z)$, we truncate the series representations(just as above) of the kernel functions for the operators $\mathcal{Q}$ and $\mathcal{R}$ as well as employ a standard $64$ point Riemann sum approximation for the integrals. This gives $64 \times 64$ discretization of the operators denoted ${\bf Q}$ and ${\bf R}$. To visualize the scatterer we let plot 
$$ W_{\text{FF}}(z) = \Big|\Big({\bf Q} {\bf N} {\bf Q}^\top {\bf R} \,  \phi_z, \phi_z \Big)_{\ell^2}\Big|^{p_1} \quad  \textrm{ with } \quad  \phi_z =[ \text{e}^{-\text{i}k z \cdot \hat{x}_1} , \cdots , \text{e}^{-\text{i}k z \cdot \hat{x}_{64}} ]^\top. $$ 
Here ${\bf N} = \big[ u^{s,\delta} (x_i , y_j) \big]_{i,j = 1}^{64}$ and $p_1$ is a positive parameter to sharpen the resolution of the image(see for e.g. \cite{Liu}). 

Now for the imaging functional $W_{\text{TDSM}}(z)$ we must construct the polynomial approximate $P_{\alpha,\ep}(t)$ for the filter function $\Gamma_{\alpha} (t)$ defined in the previous section. Here we proceed just as in \cite{DSMHarris} where ${\color{black} P_{\alpha,\ep}}(t)$ is constructed such that 
\[{\color{black} P_{\alpha,\ep}}(t) =    \sum\limits_{m=1}^{3} c_m t^m \quad \text{ such that } \quad P_{\alpha,\ep}(t_\ell) = \frac{ \sqrt{t_\ell}}{\alpha + t_\ell} \]
with $t_\ell$ is given by the 10 equally spaced point in the interval $[0 , \|{\bf Q} {\bf N} {\bf Q}^\top {\bf R} \|_2]$. In all our examples, we fix the parameter $\alpha = 10^{-3}$ since in-general we want the filter function to be an approximation of $1/\sqrt{t}$. Therefore, we have that the imaging functional $W_{\text{TDSM}}(z)$ can be numerically approximated using the singular value decomposition of ${\bf Q} {\bf N} {\bf Q}^\top {\bf R}$. Indeed, by following \cite{DSMHarris} we have that the discretization of the imaging functional is given by 
$$W_{\text{TDSM}}(z)=  \left| \sum\limits_{j=1}^{64} P^2_{\alpha,\ep}(s_j )  \big| ( {\bf v}_j , \phi_z )_{\ell^2} \big|^2 \right|^{p_2}$$
where $(s_j , {\bf v}_j ) \in \R_{> 0} \times \C^{64}$ are the singular values and right singular vectors of ${\bf Q} {\bf N} {\bf Q}^\top {\bf R}$. Again, $p_2$ is a positive parameter to sharpen the resolution. For the last imaging functional $W_{\text{CD}}(z)$, we approximate the integrals using a standard $64$ point Riemann sum approximation in each variable. In each case we normalize the imaging functionals to take $1$ as their maximal values. Therefore, we have that the imaging functionals should be approximately 1 in $D$ or on the boundary $\partial D$ and will take small values on the exterior of $D$.  \\

\noindent{\bf Example 1: Circular Scatterer} \\
For this example we plot the imaging functionals to recover a circle. Therefore, we have that the radial function describing the scatterer is given by  $r(\theta) =0.5 $. Here we take the wave number $k=4$ as well as the parameters $p1=p2=4$ and $\rho=8$. The dotted line represents the actual boundary of the circular obstacle in Figure \ref{recon1}.\\
\begin{figure}[H]
\hspace{-0.6in}{\includegraphics[width=1.2\linewidth]{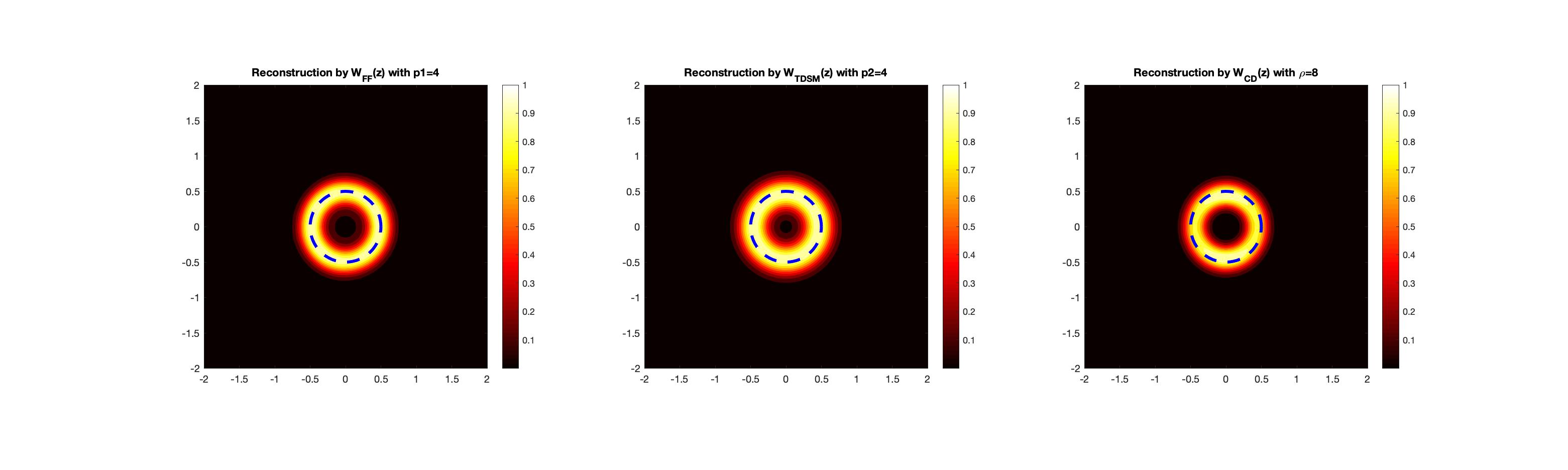}}  
\caption{The reconstruction of the circular obstacle by the three direct sampling imaging functionals. In this example, we take $\delta=0.05$ which corresponds to a 5$\%$ noise level.}
\label{recon1} 
\end{figure}

\noindent{\bf Example 2: Acorn-Shaped Scatterer} \\
Now we present a numerical example for recover an acorn-shaped obstacle. Here, the radial function describing the scatterer is given by  $$r(\theta) = 0.25\big(2+0.5\cos(3\theta)\big).$$ For this example, we again take the wave number $k=4$ as well as $p1=p2=4$ and $\rho=8$. The dotted line represents the actual boundary of the acorn-shaped obstacle in Figure \ref{recon2}.\\
\begin{figure}[H]
\hspace{-0.6in}{\includegraphics[width=1.2\linewidth]{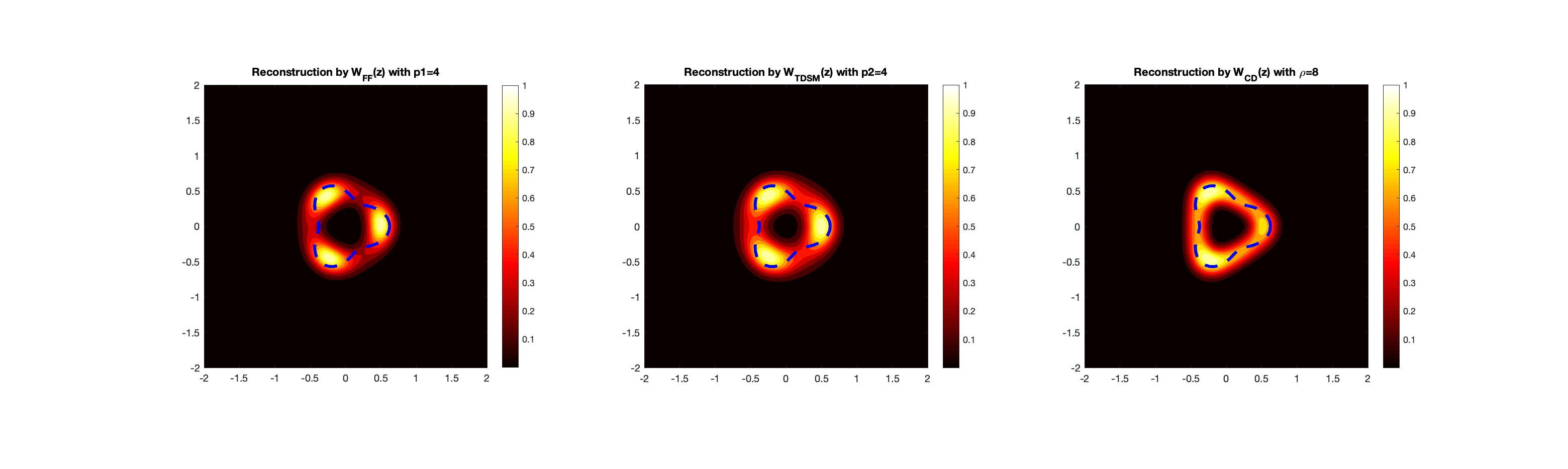}}  
\caption{The reconstruction of the acorn-shaped obstacle by the three direct sampling methods. In this example, we take $\delta=0.05$ which corresponds to a 5$\%$ noise level.}
\label{recon2} 
\end{figure}

\noindent{\bf Example 3: Flower-Shaped Scatterer} \\
Now we present a numerical reconstruction for a flower-shaped obstacle where the radial function is given by  
$$r(\theta) = 0.75\big(1-0.25\sin(4\theta)\big).$$ 
For this example we again take the wave number $k=4$ as well as $p1=p2=4$ and $\rho=8$. The dotted line represents the actual boundary of the flower-shaped obstacle in Figure \ref{recon3}.\\
\begin{figure}[H]
\hspace{-0.6in}{\includegraphics[width=1.2\linewidth]{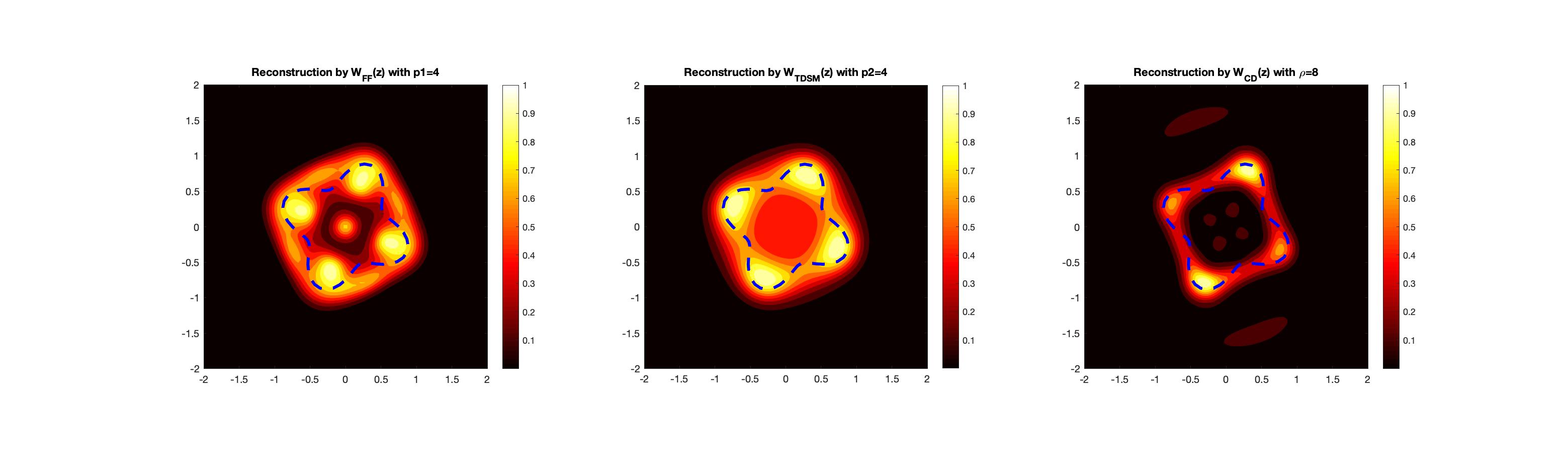}}  
\caption{The reconstruction of the flower-shaped obstacle by the three direct sampling methods. In this example, we take $\delta=0.05$ which corresponds to a 5$\%$ noise level.}
\label{recon3} 
\end{figure}

\noindent{\bf Example 4: Rounded-Square Scatterer} \\
Now we present a numerical reconstruction for a rounded-square shaped obstacle where the radial function is given by  
$$r(\theta) = 0.5 \Big( |\sin(\theta)|^{10} +  \frac{1}{10} |\cos(\theta)|^{10} \Big)^{-1/10}.$$ For this example, we again take the wave number $k=4$ as well as $p1=p2=4$ and $\rho=8$. The dotted line represents the actual boundary of the rounded-square in the figures. Here we provide the reconstruction for two different noises levels in  Figures \ref{recon4} and  \ref{recon5}. \\
%This is to see that these methods seem to be highly stable with respect to noise in the given data given in  Figures \ref{recon4} and  \ref{recon5}. 
\begin{figure}[ht]
\hspace{-0.6in}{\includegraphics[width=1.2\linewidth]{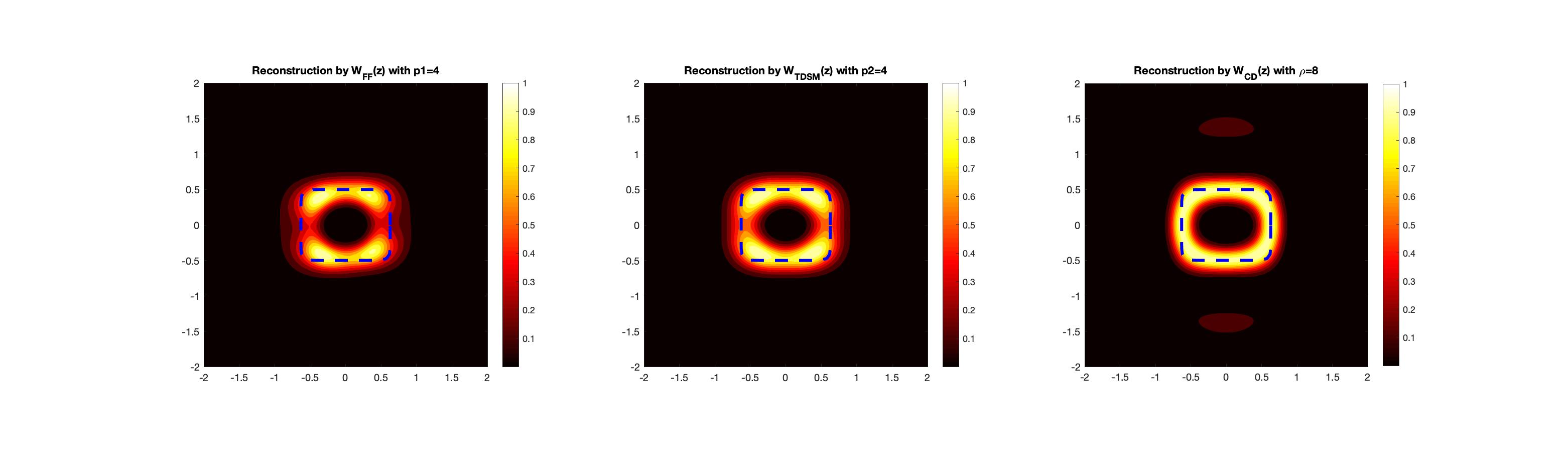}} 
\caption{The reconstruction of the rounded-square by the three direct sampling imaging functionals. In this example, we take $\delta=0.05$ which corresponds to a 5$\%$ noise level.}
\label{recon4} 
\end{figure}
\begin{figure}[ht]
\hspace{-0.6in}{\includegraphics[width=1.2\linewidth]{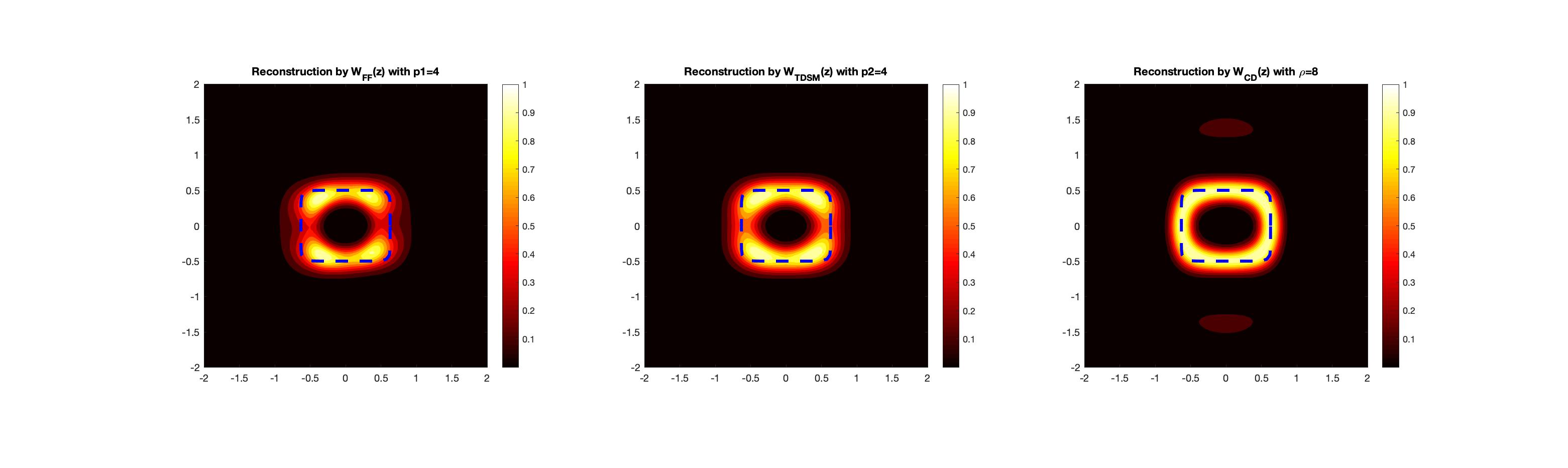}} 
\caption{The reconstruction of the rounded-square by the three direct sampling imaging functionals. In this example, we take $\delta=0.25$ which corresponds to a 25$\%$ noise level.}
\label{recon5} 
\end{figure}

In the presented examples we see that each imaging functional is stable with noise added to the data. This has been noticed in \cite{Liu} where large amounts of noise was added to the data which does not seem to affect the reconstruction much. As we see in Figures \ref{recon4} and \ref{recon5}, there is little to no difference in the reconstructions when the noise level is increased. We also see that the reconstructions are similar in many cases which is due to the fact that with the choice of parameters we have that the three imaging functionals should have the same decay rate as $\text{dist}(z,D) \to \infty$. Next, we present two examples with partial aperture data. Notice, the imaging functionals discussed here require full aperture data on the measurement surface $\Gamma$.\\

\noindent{\bf Example 5: Partial Aperture Data } \\
Here we provide two examples of reconstructing the rounded-square shaped obstacle with partial aperture data. Again, we will take the wave number $k=4$ as well as $p1=p2=4$ and $\rho=8$ where we use the imaging functionals with data only given on $3/4$ and $1/2$ of the measurement surface. \\
\begin{figure}[H]
\hspace{-0.6in}{\includegraphics[width=1.2\linewidth]{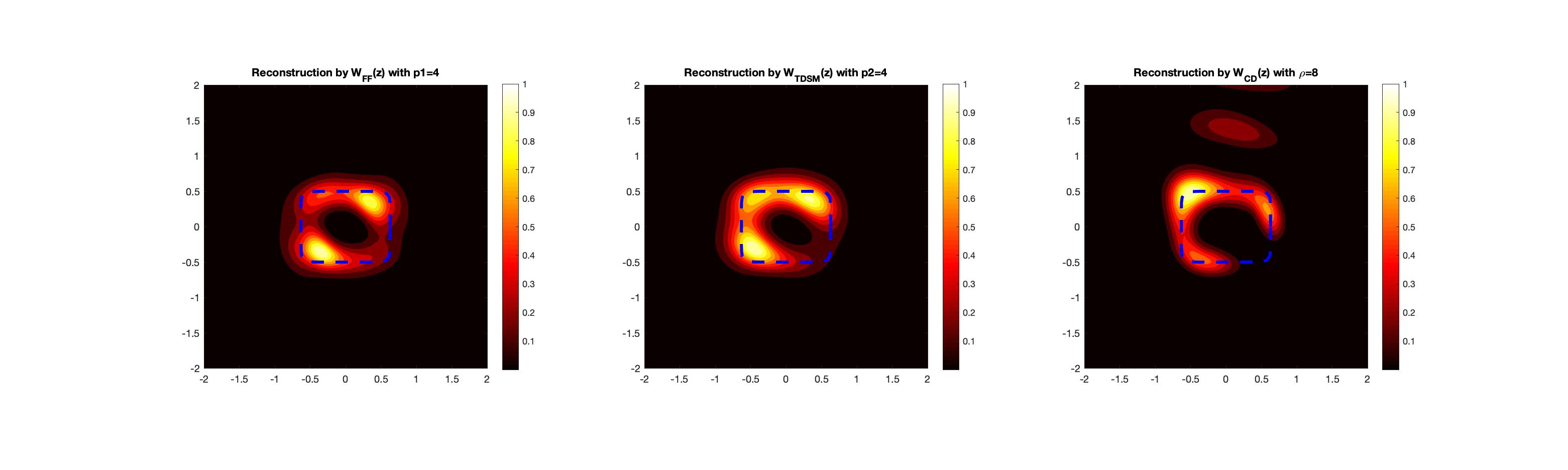}} 
\caption{The reconstruction of the rounded-square with $3/4$ partial aperture data i.e. measurements only taken on $\theta_j \in [0 , 3\pi/2 )$. In this example, we take $\delta=0.05$ which corresponds to a 5$\%$ noise level.}
\label{partialrecon1} 
\end{figure}
\begin{figure}[H]
\hspace{-0.6in}{\includegraphics[width=1.2\linewidth]{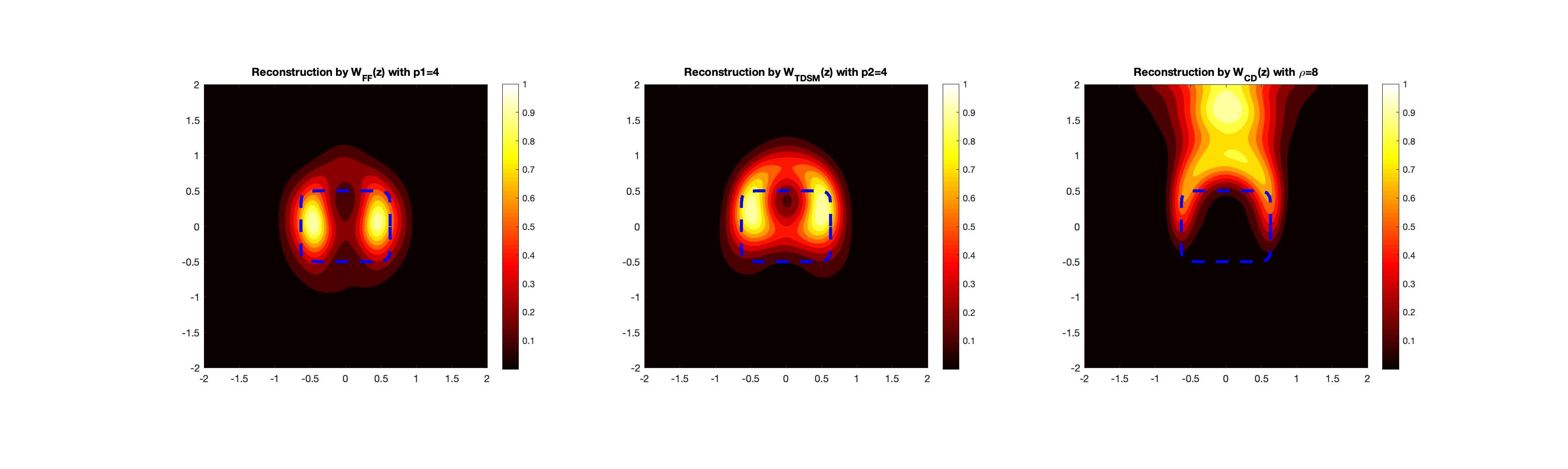}} 
\caption{The reconstruction of the rounded-square with $1/2$ partial aperture data i.e. measurements only taken on $\theta_j \in [0 , \pi )$. In this example, we take $\delta=0.05$ which corresponds to a 5$\%$ noise level.}
\label{partialrecon2} 
\end{figure}

In Figures \ref{partialrecon1} and \ref{partialrecon2} we see that $W_{\text{FF}}(z)$ and $W_{\text{TDSM}}(z)$ seem to give better reconstructions than $W_{\text{CD}}(z)$ for partial aperture data. Recently, in \cite{datacomppaper,LiuSun} some data completion methods are used to compute the missing scattering data which is then used by a qualitative method to recover the scatterer. Applying these data completion methods can possibly be employed to provide better reconstructions with partial aperture data. Also, to derive a theoretically valid estimate for the imaging functions with partial aperture data one should be able to use Theorem 4.1 in \cite{partial}.

%%%%%%%%%%%%%%%%%%%%%%%%%%%%%%%%%%%%%%%%%%%%%%%%%%%%%%%%%%%
\section{Conclusions}\label{end}
In this paper, we have developed new direct sampling methods for near-field measurements. We have focused on the case where the scatterer is a sound-soft obstacle but just as in \cite{postdocpaper} we see that these imaging functionals should work for other types of scatterers. This is one of the main advantages of direct/qualitative reconstruction methods but these methods do require full aperture data for their theoretical justification. Our numerical experiments seem to suggest that the imaging functionals derived by a far-field transformation provide reasonable results with partial aperture data. A direction that this research can progress is to develop theoretical justification for the resolution analysis for new direct sampling methods with partial aperture data. Also, just as in \cite{multifreq1,multifreq2} one can study the problem with multi-frequency data which can often help reduce the amount of sources and receivers. One would need to factorize the corresponding multi-frequency data operator as is done in \cite{FM-multifreq}. \\

\noindent{\bf Acknowledgments:} The research of I. Harris is partially supported by the NSF DMS Grant 2107891.

%%%%%%%%%%%%%%%%%%%%%%%%%%%%%%%%%%%%%%%%%%%%%%%%%%%%%%%%%%%

\end{document}